\documentclass{amsart}
\usepackage{amscd, amsfonts, amsmath, amsthm, amssymb}
\usepackage{tabularx}
\usepackage{longtable}
\usepackage{hhline}
\usepackage{enumerate}
\usepackage{booktabs}
\usepackage{comment}
\usepackage
[  backref,
letterpaper = true,
hyperindex = true,
linktocpage = true
]
{hyperref}

\newcommand{\Z}{\mathbb{Z}}
\newcommand{\F}{\mathbb{F}}

\theoremstyle{plain}
\newtheorem{theorem}{Theorem}[section]
\newtheorem{conjecture}[theorem]{Conjecture}

\theoremstyle{definition}

\newtheorem*{remark*}{Remark}
\newtheorem*{remarks*}{Remarks}

\numberwithin{table}{theorem}

\begin{document}

\author{Jos\'e Alejandro Lara Rodr\'iguez and Dinesh S. Thakur}
\thanks{The authors supported in part by PROMEP grant F-PROMEP-36/Rev-03
SEP-23-006 and
by NSA grant  H98230-13-1-0244 respectively}

\title[Zeta-like Multizeta Values]{Zeta-like 
Multizeta Values for $\F_q[t]$}

\address{Department of Mathematics, University of Arizona, Tucson, AZ 85721 USA,
thakur@math.arizona.edu and Department of Mathematics, University of Rochester, 
Rochester, NY 14627 USA, dinesh.thakur@rochester.edu}

\address{
Facultad de Matem\'aticas, Universidad Aut\'onoma de Yucat\'an, Perif\'erico
Norte, Tab. 13615,
M\'erida, Yucat\'an, M\'exico, lrodri@uady.mx}

\dedicatory{Dedicated to the genius of Srinivasa Ramanujan}


\date{\today}

\begin{abstract}
We prove and conjecture several relations between multizeta 
values for $\F_q[t]$, focusing on zeta-like values, namely those 
whose ratio with the zeta value of the same weight is rational 
(or equivalently algebraic). In particular, we describe them 
 conjecturally fully for $q=2$, or more generally for 
 any $q$ for `even' weight (`eulerian' tuples).  We provide some data in support 
of the guesses. 
\end{abstract}

\maketitle

\section{Introduction}

Relations between multizeta values defined by Euler have been investigated 
extensively for the last two decades, and the conjectural forms of 
these relations have  many structural connections with several interesting areas 
(see \cite{Ct, Z} and references there) of mathematics. 
In some sense, the relations have been at least conjecturally 
understood, though much remains to be proved and relating the general 
framework to specific instances 
is often hard. 

We will look at the function field analog \cite{T, AT, Tm, Tbanff, LR, L12}, where
the relations 
are still not conjecturally understood, though in contrast, there 
are also some very strong results transcendence and independence 
results \cite{ CY, Ch, CPY} proved. 

While for the Euler multizeta values, the relations come via comparing the 
two families of shuffle relations, in our function field setting, there is 
only one shuffle family \cite{Ts}. While the rational number field is the prime 
field in characteristic zero giving coefficients for the relations, 
in the function field case the prime field is not the analogous rational function 
field, but just the finite field, which does not see all the relations. Some such relations 
were proved in \cite{Tm, LR}. 

While second author's student George Todd is doing extensive numerical study 
of general relations using an analog of the `LLL method', in this paper 
we focus on the two term relations of special type, namely 
zeta-like multizeta, i.e., those whose ratio with the (Carlitz) zeta 
value of the same weight is rational (or equivalently algebraic). 
We provide several results, and conjectures, with full conjectural description for 
$q=2$, or more generally for any $q$ with `even' weight (`eulerian' tuples).

We  first fix the notation and give the basic definitions.  
Next we summarize the known and the new results on zeta-like values, 
and state the conjectures. Then we give the proofs 
of the results. 
 Finally we discuss the  numerical data, calculated by the first author,
  giving some evidence for the  conjectures made from it.    

\pagebreak

\section{Notation and Basic definitions}

\begin{tabular}{l l}
$\Z$  & \{integers\},\\
$\Z_+$ &  \{positive integers\},\\
$q$ &  a power of a prime $p$, $q = p^s$,\\
$\F_q$ & a finite field of $q$ elements,\\
$A$ & the polynomial ring $\F_q[t]$, $t$ a  variable\\
$A_+$ &   monics in  $A$,\\
$K$ &   the function field $\F_q(t)$,\\
$K_\infty$ & $\F_q((1/t)) = $ the completion of $K$ at $\infty$,\\
$A_{d^+}$ & $\{\mbox{elements of } A_+ \mbox{of degree }d\}$, \\
$[n]$ &   $ t^{q^n}-t$,\\
$\ell_n$ & $ \prod _{i= 1} ^n (t-t^{q^i})= (-1)^n L_n = (-1)^n [n][n-1]\dotsm
[1]$, \\
`even' & $  \mbox{multiple of } q-1$,\\
\end{tabular}

We first recall definitions of power sums, iterated power sums, zeta 
and multizeta values \cite{T, Tm}. 

For $s\in \Z_+$ and $d\geq 0$, write
\begin{align*}
S_d(s) := \sum _{ a\in A_{d^+}} \frac{1}{a^{s}} \in K.
\end{align*}
(This is  $S_d(-s)$ in the notation of \cite{T}.)

Given integers $s_i\in \Z_+$ and $d\ge 0$ put
\begin{align*}
  S_d(s_1,\dotsc,s_r) = S_d(s_1) \sum _{d > d_2 >\dotsb > d_r \ge 0}
S_{d_2}(s_2) \dotsm S_{d_r}(s_r) \in K.
\end{align*}

For $s_i\in\Z+$, we define multizeta values 
\begin{align*}
\zeta(s_1,\dotsc,s_r):= \sum _{d_1>\dotsb >d_r\ge 0} S_{d_1}(s_1) \dotsm
S_{d_r}(s_r)= \sum \frac{1}{a_1^{s_1}\dotsm a_r^{s_r} }\in K_\infty,
\end{align*}
where the second sum is over all $a_i\in A_+$ of degree $d_i$ such that
$d_1>\dotsb >d_r\ge 0$. We say that this multizeta value (or rather the tuple
$(s_1, \dotsc, s_r)$)
has depth $r$ and
weight  $\sum s_i$. In depth one, we recover the Carlitz zeta. 

We refer to  \cite{C1, G, T} for background on this and general 
function field analogies. Carlitz proved  
analog of Euler's result that for `even' $s$, $\zeta(s)$ 
is rational multiple of $\tilde{\pi}^s$, where the Carlitz period 
$\tilde{\pi}$ is analog of $2\pi i$.

A multizeta value $\zeta(s_1, \dotsc, s_r)$ of depth $r$ (or  the $r$-tuple
$(s_1, \dotsc, s_r)$) is \emph{zeta-like} if the ratio
\begin{align*}
\zeta(s_1, \dotsc, s_r)/\zeta(s_1 + \dotsb + s_r)
\end{align*}
is rational. (We always use depth $r>1$ below, sometimes 
without mention, because  in the $r=1$ case everything is  
 zeta-like by definition). 
A multizeta value of weight $w$ is called \emph{eulerian}, if it is a
rational multiple of $\tilde{\pi}^w$. So eulerian is a special
case of zeta-like for `even' weight, by Carlitz result (mentioned 
above) which says that  in depth one,
all the zeta values of `even' weight are eulerian.

A strong transcendence result \cite{Ch} proved in the function 
field case shows that if the ratios in the definition are not 
rational, they are not even algebraic and in fact the multizeta 
and corresponding zeta values are then algebraically independent. 
Another strong transcendence result \cite{CY} shows that 
the Carlitz zeta value of not `even' weight and $\tilde{\pi}$ are algebraically independent. 

Since $\zeta(ps_1, \dotsc, ps_r)=\zeta(s_1, \dotsc, s_r)^p$, in all
the discussion we can restrict to tuples where not all $s_i$'s are divisible 
by $p$. We call such tuples \emph{primitive}.

\section{Old and new results on Zeta-like values}

For the Euler multizeta in the number field case, the classical sum shuffle 
relation immediately implies that $\zeta(2n, 2n)$ are eulerian, and combined with the 
usual transcendence conjectures, it implies that $\zeta(2n+1, 2n+1)$ are not 
zeta-like. In the function field case, the classical sum shuffle relation 
does not hold in general, so we only know by this method \cite[Thm. 5.10.6]{T} that, when $p\neq 2$,  
that  $\zeta(kp^n, kp^n)$ is not zeta-like, 
if $2k\leq q$ (\cite{CY} giving the required transcendence result). 
(Another instance of different shuffle \cite[Thm. 6.3]{L12} similarly 
shows that $\zeta(q^n-1, q^n)$ is not zeta-like, for $q>2$). 
In \cite{T, Tm, LR}, more  examples of zeta-like and non-zeta-like values 
of `even' and `odd' weights were proved. Combining with general shuffle 
relations \cite{Ts}, some more such results can be proved. 
But we have now proved 
much stronger results, which we will recall below. 

In \cite{CPY}, using the interpretation \cite{AT} of multizeta 
values as periods of iterated extensions of tensor powers 
of Carlitz-Anderson $t$-motives, it was proved that if
$\zeta(s_1, \dotsc , s_r)$ is zeta-like (eulerian in
the first version), then $\zeta(s_2, \dotsc, s_r)$ is
eulerian, so that all $\zeta(s_k, \cdots, s_r)$ are eulerian
and $s_i$ are `even', for $i\geq 2$. (See \cite[5.3]{Tm}). 

{\bf Remark} This implies some, but not all, of the 
non-zeta-like special results mentioned above. Many can be proved 
by direct appeal to \cite{CY} and shuffle and other results proved, 
a few (such as $\zeta(2, 1)$ is not zeta-like for $q=2$) 
 were proved \cite[Thm. 5.10.12]{T} without using \cite{CY}. 

While \cite{CPY} was being proved for the eulerian case, 
we had conjectured this (and a few more implications) for zeta-like case, but 
only in depth 2 and were starting calculations 
in general depth, which give many interesting conjectural 
restrictions recalled below. 

We now state some families of zeta-like (so eulerian, if 
the weight is `even') multizeta values of depth two. 
 Proofs will be given in the fifth section. 

\begin{theorem}\label {main} For any (prime power) $q$, we have 
\begin{align}
\label{main1}
\zeta(q^n- \sum _{i=1}^s q^{k_i}, (q-1)q^n)
= \frac{(-1)^s}{\ell_1^{q^n}} \prod _{i=1}^s [n-k_i]^{q^{k_i}}
\zeta(q^{n+1} - \sum _{i=1}^s q^{k_i}),
\end{align}
where $n>0$, $1 \le s <q$, $0 \le k_i <n$.

Let $n\geq 0$, $0\le k_i\le n+1$, $1\leq s_1\le q$, $0\leq s_2\leq q-s_1$. 
Then for $a = s_1 q^n$ and $b = s_1(q^{n+1}-q^n) + \sum_{i=1}^{s_2}
(q^{n+1}-q^{k_i})$, we have 
  \begin{align}
\label{main2}
\zeta(a, b) = \frac{1}{\ell_1^{s_1q^n}}\zeta(a+b).
\end{align}

\begin{align}
\label{main3}
\zeta(q^2-(q-1), (q-1) (q^{2}+1))
&=
\frac{1-[2]^q}{\ell_1^{q^2-1} \ell_2}
\zeta(q^3),\\
\label{main4}
\zeta(2q-1, (q-1)(q^2+q-1)) & = \frac{1-[2]^q}{\ell_1^{q+1}
\ell_2^{q-1}}
\zeta(q^3).
\end{align}

\begin{align}
\label{main5}
\zeta(1,q^2-1) & = \zeta(q^2)(1/\ell_1 + 1/\ell_2).
\end{align}

For $q>2$, $n\geq 0$ and $-1 \le j \le n$,

\begin{align}
\label{main6}
\zeta((q-1)q^{n} -1, (q-1)q^{n+1} + q^{n}-q^{n-j}) =
-\frac{[n+1]}{[1]^{(q-1)q^n}} \zeta(q^{n+2}-q^{n-j}-1).
\end{align}
\end{theorem}

Next we state a theorem (proved in section 5) giving zeta-like family of arbitrary depth. 

\begin{theorem}\label{mainII}
For any $q$,
\begin{align*}
\zeta(1,q-1, (q-1)q, \dotsc, (q-1)q^n)
=\frac{(-1)^{n+1}}{
[1]^{q^n} [2]^{q^{n-1}} \dotsm [n+1]^{q^0}
}
\zeta(q^{n+1}).
\end{align*}
\end{theorem}

\section{Observations, Guesses and Conjectures}

Now we state some conjectures (based on the numerical data
and on consistency with the theorems and the proof methods)  
with varying degrees of confidence and evidence! 

\begin{conjecture}{\bf Tuple restrictions}  
 If $(s_1,\dotsc, s_r)$ is zeta-like, then
\begin{enumerate}
\item $s_i \le s_{i+1}$, $i = 1,  \dotsc, r-1$. Furthermore,
$(q-1) s_i \le s_{i+1} \le (q^2-1)s_i$.
\item $(s_2, \dotsc, s_r)$ is eulerian and $(s_1, \dotsc, s_{r-1})$ is zeta-like
\end{enumerate}
\end{conjecture}

Note that part 2 can be iterated and implies that $s_i$ are `even' for $i\geq 2$ (already proved 
 together with the first part of (2), in \cite{CPY}, as mentioned before). 

\begin{conjecture} {\bf Splicing of tuples when $q=2$}
If $(s_1, \dotsc, s_k)$ and $(s_k,\dotsc,  s_r)$ are zeta like
and the total weight 
$\sum_{i=1}^r s_i$ is a power of $2$ or a power of $2$ minus one, then $(s_1,
\dotsc, s_r)$ is zeta-like, except when the two tuples to be spliced are
(1,1) and (1,1). 
\end{conjecture}

{\bf Remarks} We have not seen any more failures in the limited data 
we have.  We are investigating the situation  
for general $q$, where splicing conditions seem to be much more 
restrictive and seem to depend (in the limited data we have) 
on combinatorics of digit expansions. But splicing of eulerian tuples 
seems to work in weights $q^n-1$. 

\begin{conjecture} {\bf Weight restrictions}  
\begin{enumerate}
\item Eulerian multizeta value (in depth $r>1$) can occur only in weights $p^m(q^k-1)$, with primitive 
ones only in weights $q(q-1)$ or $q^n-1$ for $q>2$ and in weights $2^n-1$ and $2^n$, if $q=2$. 
\item When $q = p$, depth $r>1$, the weight of zeta-like but non-eulerian tuple is
$p^m$ times number with no zero digit and at most one 1 digit.
\item In depth $r$, the smallest weight of zeta-like value is $q^{r-1}$.
\item When $q>2$, the smallest weight of eulerian value is 
$q^r-1$,  $(q-1)q$, and $q-1$ according as depth $r>2$, $r=2$ 
and $r=1$ respectively.  
\item   Weight $q^k$ is not a zeta-like weight 
of a primitive tuple,  if 
$k>r>3$, and also when $k>3$ if $r=3$, or $2$, where $r$ is the depth.  
\end{enumerate}
\end{conjecture}

\pagebreak

{\bf Remarks} (0) Let us recall the known results for the Euler multizeta.
(We use the standard short-form $\{X\}_k$ standing for the tuple $X$ repeated 
$k$ times.) 
Euler proved that $\zeta(3, 1)=\zeta(4)/4$ and 
$\zeta(2, 1)=\zeta(3)$, which generalizes to zeta-like 
 $\zeta(2, \{1\}_k)=\zeta(k+2)$ (special case of Hoffman-Zagier duality relation) 
 and $\zeta(\{3, 1\}_k)=\zeta(\{2\}_{2k})/(2k+1)$
(Broadhurst's result, conjectured by Zagier) which is known to be eulerian as  
$\zeta(\{2\}_k)=\pi^{2k}/(2k+1)!$ 
(see e.g., \cite{Z}). In fact, $\zeta(\{2n\}_k)$ is also eulerian. 
(Proof:  Let the  induction hypothesis 
$P(k)$ be that $A_k:=\zeta(\{2n\}_k)$ and $B_{k, m}:=\sum \zeta(X(i))$ are eulerian 
for all $n, m$, where 
$X(i)$ runs through all length $k$ tuples with $k-1$ entries $2n$ and one entry 
$2nm$. The sum shuffle gives $\zeta(2n)A_k=(k+1)A_{k+1}+B_{k, 2}$ and 
$\zeta(2mn)A_k=B_{k+1, m}+B_{k, m+1}$ proving the result by induction.
For a proof using generating functions, see \cite{BBB}. We thank J. Zhao for 
the reference). 
 In our very limited numerical search 
(weight $\le 50$ for depth 2, and even lower for depths 3, 4), 
as well as limited search of the vast literature, we did not find any other 
zeta-like tuples. We do not know whether there are any more examples, 
or conjectures based on theoretical or numerical evidence.

For Euler's multizeta, all even weights $>2$ are eulerian 
in depth more than one, as $\zeta(2_k)$ is eulerian. 
 In our case, for $q=3$, even $\zeta(2, 2)$
is not eulerian by \cite[Thm. 5.10.12]{T}, and this conjecture 
predicts much stringent weight conditions. 
It is conjectured for the Euler multizeta that the eulerian case occurs only in even weights. 
In our case, we know by \cite{Ch} that the eulerian case occurs 
only in `even' weights. For the Euler's multizeta, 
$\zeta(2n, 2n)$ are eulerian  
of weight $4n$ and depth 2, though weight $4n+2$ does not seem to be eulerian weight in 
depth $2$.  
 
(1) The weight $p^m(q^k-1)$, with $m>0$ for eulerian value 
can occur with primitive tuples, e.g., $q=2$ and $(1, 1), (1, 3), (3, 5)$ 
or $q=3$ and $(2, 4)$. 

(2) The parts 3 and 4  are known for depth $1$, and the occurrence 
in predicted weights is either proved in our main theorem or 
also follows from the higher depth families conjectures below. So 
the `smallest' is the real conjectural part. More data may allow 
to conjecture the depth dependence of possible $m$ and $k$ in the 
first part. 

(3) While the Theorem~\ref{main} shows that weights $q, q^2, q^3$ 
are zeta-like weights for any $q$, the last part suggests 
that weight $q^4$ is not a weight of a zeta-like tuple (in depth 
more than one, of course).  

\begin{conjecture} {\bf Depth 2, weight at most $q^2$}
All zeta-like primitive tuples of weight at most $q^2$ and depth $2$ 
  are exactly $(i,j(q-1))$, $i = 1,\dotsc, q$, $j = i, \dotsc,  \lfloor
(q^2-i)/(q-1) \rfloor$ ($\lfloor (q^2-i)/(q-1) \rfloor$ equals $q+1$ or $q$,
depending if $i=q$ or $i<q$):
\begin{align*}
\begin{array}{*{6}{c} }
(1, q-1) & (1, 2(q-1)) & (1,3(q-1))& (1,4(q-1)) &\dotsc & (1,(q+1)(q-1))\\
& (2, 2(q-1)) & (2, 3(q-1)) & (2,4(q-1)) & \dotsc & (2, q(q-1))\\
& & (3, 3(q-1)) & (3, 4(q-1)) & \dotsc & (3, q(q-1))\\
& & & & & \vdots\\
& & & & & (q, q(q-1))
\end{array}
\end{align*}
\end{conjecture} 

Note that our theorems imply that those tuples are zeta-like. The converse is 
the conjectural part. 

\begin{conjecture} {\bf $q=2$, Depth $2$}
Let $q=2$, the zeta-like (eulerian equivalently) primitive tuples of depth two are 
exactly $(1, 1), (1, 3), (3, 5)$ and $(2^n-1, 2^n)$, $(2^n, 2^{n+1}+2^n-1)$. 
\end{conjecture} 

Again, our theorems imply all of these are eulerian, the converse is conjectural. 
(This seems to be true up to weight 128 from the numerical data). Note that 
4.1, 4.2, 4.3 part (1) and 4.5 conjecturally completely describe all eulerian 
tuples, if $q=2$, and remark after 4.2 reduces the general $q$ eulerian  case to depth 2. 
We are trying to verify the guess that all the depth 2 primitive eulerian tuples
are exactly (covered by Theorem~\ref{main} )$\zeta(q-1, (q-1)^2)$, $\zeta(q^n-1, (q-1)q^n)$ 
and $\zeta(q^n(q-1), q^{n+2}-1-q^n(q-1))$, for $q>2$. (These miss (1, 3), (3, 5) for $q=2$).

\begin{conjecture}{\bf Conjectural zeta-like families of arbitrary depth}
\begin{enumerate} 
\item 
For any $q$, $n \ge 1$ and $r \ge 2$, we have
\begin{align*}
\zeta(q^n-1, (q-1)q ^n, \dotsc, (q-1)q^{n+r-2})
=
\frac{[n+r-2] [n+r-3] \dotsm [n]}
{[1]^{q^{n+r-2}} [2]^{q^{n+r-3}} \dotsm [r-1]^{q^n}  }
\zeta(q^{n+r-1}-1).
\end{align*}

\item
For any $q$, $n\geq 0$,
\begin{align*}
\zeta(1, q^2-1, (q-1)q^2, \dotsc, (q-1)q^{n+1})=
\frac{[n+2]-1}{\ell_1[n+2]}\frac{1}
{
 \ell_1^{(q-1)q^n} \ell_{2}^{(q-1)q^{n-1}}
\dotsm \ell_{n-1}^{(q-1)q^2}\ell_n^{q^2}
}
\zeta(q^{n+2}).
\end{align*}

\item
For $q>2$, $n \ge 0$ and $r \ge 2$,
\begin{align*}
\zeta((q-1)q^n-1, (q-1)q^{n+1}, \dotsc, (q-1)q^{n+r-1})
\end{align*}
equals
\begin{align*}
\frac{(-1)^{r+1}
[n+r-1] [n+r-2] \dotsm [n+1]
}
{
[1]^{ (q^{r-1}-1)q^n }
[2]^{ (q^{r-2}-1)q^n }
\dotsm
[r-1]^{ (q-1)q^n }
}\zeta(q^{n+r}-q^n-1).
\end{align*}
\end{enumerate}
\end{conjecture}

It seems quite likely that these families can be proved by a proof similar 
to that of Theorem~\ref{mainII} below, but this has not been carried out yet. Note also 
that in the depth 2 case, all of these are proved in Theorem~\ref{main}. (Here the 
second one reduces for $n=0$ to \ref{main5} by usual conventions on empty products, 
 sums, patterns and indexing). 

\section{Proofs}

The following formulas, which are consequence of Theorems 1 and 3 in  \cite{LT},
will be used in the proof of the main
theorem.
\begin{enumerate}
\item For $1 \le s < q$ and $0 \le k_i <k$
with $1 \le i \le s$, we have
\begin{align}
\label{Sd}
S_d(q^k- \sum_{i=1}^s q^{k_i})& = \ell_d^{(s-1)q^k} S_d(q^k-q^{k_1}) \dotsm
S_d(d,q^k -q^{k_s}).
\end{align}

\item For $1\leq s \leq q$, and
any $0\leq k_i\leq k$, with $1\leq i\leq s$, we have
\begin{align}
\label{Sltd}
S_{<d} ( \sum _{i=1}^s(q^k-q^{k_i}) )
=
\prod _{i=1}^s S_{<d}(q^k -q^{k_i}).
\end{align}
\end{enumerate}

We also recall Carlitz' evaluations (see e.g., \cite[3.3.1, 3.3.2]{Tm})

\begin{align}
\label{Sda}
S_d(a)=1/\ell_d^a, \ \ (a\leq q)
\end{align}
\begin{align}
\label{Sdqjminus1}
S_d(q^j-1) = {\ell_{d+j-1}}/{\ell_{j-1} \ell_d^{q^j}}
\end{align} 

\begin{align}
\label{Sltdqjminus1}
S_{<d}(q^j-1) = { \ell_{d+j-1}  }/{ \ell_j \ell_{d-1}^{q^j}},
\end{align} 

\begin{proof}[Proof of Theorem~\ref{main}]
Let $a = q^n- \sum _{i=1}^s q^{k_i}$ and $b = (q-1)q^n$.
By definition, we have
$$\zeta(a,b)
= \sum _{d=1}^\infty S_d(a,b)
=\sum _{d=1}^\infty S_d(a)S_{<d}(b).$$
Using \eqref{Sd},  \eqref{Sltd}, \eqref{Sdqjminus1} and \eqref{Sltdqjminus1}, 
by straight calculations we get
$$
S_d(a)S_{<d}(b) = \frac{(-1)^s}{\ell_1^{q^n}} \prod _{i=1}^s
[n-k_i]^{q^{k_i}} S_{d-1}(a+b).
$$ 
By summing over $d$ the claim \eqref{main1} follows.
The proofs of claims   \eqref{main2} and \eqref{main6} are similar, once we note
that for
\eqref{main2} we have
\begin{align*}
a + b = q^{n+2} - (q-s_1-s_2)q^{n+1} - \sum _{i=1}^{s_2} (q^{n+1}-q^{k_i}),
\end{align*}
and for \eqref{main6},  the requirement $q>2$ guarantees
that the formula for $S_d(q^{n+1}-q^n-1)$ can be applied.

Now, let $a = q^2-(q-1)$ and $b=(q-1) (q^{2}-q) + (q^{2}-1)$.
Using formulas \eqref{Sd} and \eqref{Sltd} again, a straight calculations yields
\begin{align*}
S_d(a, b) & =
\frac{1}{\ell_1^{q^2-1} \ell_2}
\frac{ t^q- t^{q^{d+2}} }{ \ell_{d-1}^{q^3} }.
\end{align*}
Recall that the inverse around origin of the Carlitz exponential $e_C(z)$ is
the Carlitz logarithm $\log(z) = \sum z^{q^d}/\ell_d$ and it satisfies $t
\log(z) = \log(tz) +\log(z^q)$. Therefore, $ t\log(1) = \log(t) + \log(1)$ or
equivalently
$\log(t) = (t  - 1)\log(1)$. Since $\zeta(1) = \log(1)$ and $\zeta(1)^{q^3} =
\zeta(q^3)$, by summing over $d$ we get
\begin{align*}
\sum _{d=1}^\infty \frac{ t^q- t^{q^{d+2}} }{ \ell_{d-1}^{q^3} }
& =
t^q \sum _{d=0}^\infty \frac{1}{\ell_d^{q^3}} -
\sum _{d=0}^\infty \frac{t^{q^{d+3}}}{\ell_d^{q^3}}\\
& =
t^q \log(1)^{q^3} - \log(t)^{q^3}\\
& =
t^q \zeta(q^3) - (t^{q^3}-1) \zeta(q^3)\\
& = ( -[2]^q + 1)\zeta(q^3),
\end{align*}
and claim \eqref{main3} follows.

Now, for  \eqref{main4}, let $a = 2q-1$ and $b = q^2-q + (q-1)(q^2-1)$. We have
\begin{align*}
S_d(a, b) & =
\frac{1}{\ell_1^{q+1} \ell_2^{q-1}}
\frac{(-[d+1]^q)}{\ell_{d-1}^{q^3}}.
\end{align*}
By summing over $d\ge 1$, we obtain
\begin{align*}
\sum _{d=1}^\infty \frac{t^q-t^{q^{d+2}}}{\ell_{d-1}^{q^3}}
 = (-[2]^q+1 )\zeta(q^3),
\end{align*}
and the result follows.

Finally, \eqref{main5} is proved in \cite[Thm. 5]{Tm}
\end{proof}

\begin{proof}[Proof of Theorem~\ref{mainII}]
We claim that
$$ 
S_d(1, q-1, q(q-1), \cdots, q^n(q-1))=\frac{1}{\ell_{n+1}\ell_n^{q-1}\ell_{n-1}^{q(q-1)}
\cdots \ell_1^{q^{n-1}(q-1)}}S_{d-(n+1)}(q^{n+1}).
$$
Summing  the claimed equality over $d$ proves the Theorem. 

For $n=1$ and all $d$, this is proved in \cite[3.4.6]{Tm}. We prove 
it by induction by assuming  it for $n$ replaced by $n-1$, and considering it for $n$ 
as claimed. 

For $d<n+1$ both sides are zero, and for $n=d+1$, it follows 
using (we use these often below) \eqref{Sda} for $a=1, q-1$ together with the obvious $S_d(q^nj)=
S_d(j)^{q^n}$. We  write $s_n(d):= \sum_{j=0}^{d-1} S_j(q-1, q(q-1), \cdots, q^n(q-1))$ 
and $f_n(d):= \ell_d/(\ell_{n+1}\ell_n^{q-1}\cdots \ell_1^{q^{n-1}(q-1)}\ell_{d-(n+1)}^{q^{n+1}}).$
It is enough to show that $s_n(d)=f_n(d)$ for all $d>n+1$. 
Now $s_n(d+1)-s_n(d)$ is 
$$S_d(q-1, \cdots, q^n(q-1))=S_d(1)\sum_{j=0}^{d-1}S_j(q(q-1), 
\cdots, q^n(q-1))=S_d(1)s_{n-1}(d)^q.$$
Now $s_{n-1}(d)=f_{n-1}(d)$ by induction, and 
a simple manipulation shows that $f_n(d+1)-f_n(d)=
S_d(1)f_{n-1}(d)^q$ thus completing the proof of the claim and the theorem 
by induction. 
\end{proof}

{\bf Remarks} It might be worthwhile to point out a very special low weight 
case of (2) of Theorem 3.1 that $(n, m(q-1))$ is zeta-like, if 
$1\leq n\leq q$ and $n\leq m\leq q$. 

\section{Data}

Theory of continued fractions for function fields was first developed by 
Emil Artin in his thesis. (See \cite[Chap. 9]{T} for a survey.). We use them 
to find the zeta-like values as follows. We calculate the multizeta divided 
by zeta of same weight numerically (i.e., approximation where we use 
first few degrees rather than all), and calculate its continued fraction. 
If the ratio of actual values is rational, the continued fraction thus 
calculated will be same as the continued fraction of this rational 
for the first few partial quotients and then there will be very large partial 
quotient indicating small error in approximation. We detect this and 
then we double check 
by increasing the precision that we do get  the stabilized part, followed by increasing 
partial quotient (corresponding to reducing error), followed by non-stabilized 
part.   

The calculation was done (in stages, with guesses verified with more data) over several months by programing in SAGE and using laptops 
and mainframes. In lower depths, and small weights, small $q$'s 
 the 
calculation was exhaustive (i.e., going through all tuples looking for zeta-like
values), and sometimes guesses of higher depth, weight, $q$'s 
were checked separately to some extent. For $q = 2$, depth $2$ and $3$ and weight up to $128$ and $32$, respectively,
and for $q = 3$, depth $2$ and $3$ and weight up to $81$, calculation was
exhaustive. For $q = 4, 5$,  depths 2 and 3, we  went through all weights up to
$q^3$, but assuming that $s_i$ is `even' (called restrictive search) 
for $i \geq 2$, and $s_i \le s_{i+1}$.
However, we  checked to some extent that  the tuples not satisfying the increasing
condition are not zeta-like.
Also, we decreased precision often, otherwise the calculation would have taken much more
time. 

We only list {\it primitive} tuples. The tuples marked with * are covered by the theorems.

\pagebreak

\subsection{Data for \texorpdfstring{$q=2$}{q2}} Zeta like tuples of depth 2 and weight at most $128$.
\begin{center}
\begin{tabular}{ *{7}{c}} \toprule
(1, 1)*  & (1, 2)* & (1, 3)*  & (2, 5)*   & (3, 4)* &  (3, 5)*  & (4, 11)*\\
(7, 8)* & (8, 23)* & (15, 16)* & (31, 32)* & (16, 47)* & (32, 95)* & (63, 64)*\\
\bottomrule
\end{tabular}
\end{center}

\begin{center}
$q = 2$. Zeta like tuples of depth 3, weight at most $q^5 =32$, and more.

\begin{tabular}{ *{6}{c} } \toprule
(1, 1, 2)* & (1, 2, 4) & (1, 2, 5) & (1, 3, 4) & (3, 4, 8) &
(7, 8, 16)\\
(15, 16, 32) & (31, 32, 64)\\
\bottomrule
\end{tabular}
\end{center}


\begin{center}
$ q= 2$. Some zeta like tuples of depth 4.

\begin{tabular}{*{5}{c}}\toprule
(1, 1, 2, 4)* & (1, 2, 4, 8) & (1, 3, 4, 8) & (3, 4, 8, 16)
& (7, 8, 16, 32)\\
(15, 16, 32, 64) & (31, 32, 64, 128)\\
\bottomrule
\end{tabular}
\end{center}

\begin{center}
$ q= 2$. Some zeta-like tuples of depth 5.

\begin{tabular}{*{4}{c}}\toprule
(1, 1, 2, 4, 8)* & (1, 2, 4, 8, 16) & (1, 3, 4, 8, 16)
& (3, 4, 8, 16, 32)\\
(7, 8, 16, 32, 64) & (15, 16, 32, 64, 128) &
(31, 32, 64, 128, 256)\\
\bottomrule
\end{tabular}
\end{center}

\begin{center}
$ q= 2$. Some zeta-like tuples of depth 6.

\begin{tabular}{*{3}{c}}\toprule
(1, 1, 2, 4, 8, 16)* & (1, 2, 4, 8, 16, 32) & (1, 3, 4, 8, 16, 32)\\
(3, 4, 8, 16, 32, 64) & (7, 8, 16, 32, 64, 128) & (15, 16, 32, 64, 128, 256)\\
(31, 32, 64, 128, 256, 512)\\
\bottomrule
\end{tabular}
\end{center}


\subsection{Data for \texorpdfstring{$q=3$}{q3}}
Zeta-like tuples of depth 2 and weight up to $q^4 = 81$:
 
\begin{center}

\begin{tabular}{ *{6}{c} } \toprule
(1, 2)*  &  (1, 4)*  &  (1, 6)* &  (1, 8)* &  (2, 4)*  &  (2, 6)* \\
(3, 14)* &  (3, 20)* &  (3, 22)* &  (5, 12)* &  (5, 18)*  &  (5, 20)* \\
(5, 22)*  &  (6, 20)* &  (7, 18)* &  (7, 20)* &  (8, 18)* &  (9, 44)*\\
(9, 62)*  &  (9, 68)*  &  (9, 70)* &  (15, 62) &  (17, 36)* &  (17, 54)*\\
(17, 60)* &  (17, 62)* &  (18, 62)* &  (23, 54)* &  (25, 54)* &  (26,
54)*\\
\bottomrule
\end{tabular}
\end{center}

\begin{center}
Zeta-like tuples of depth 3,  weight $\leq q^4 = 81$ and more:

\begin{tabular}{c c c c} \toprule
(1, 2, 6)*  & (1, 6, 18)  & (2, 6, 18) & (1, 6, 20) \\
(1, 8, 18) & (5, 18, 54) & (7, 18, 54) & (8, 18, 54) \\
(17, 54, 162) & (23, 54, 162) \\
\bottomrule
\end{tabular}
\end{center}

\begin{center}
Some zeta-like tuples of depth 4.
\begin{tabular}{c c c c} \toprule
(1, 2, 6, 18)* & (1, 6, 18, 54) & (2, 6, 18, 54) &
(1, 8, 18, 54)\\
(5,18, 54, 162) & (7, 18, 54, 162) &(8, 18, 54, 162)
& (17, 54, 162, 486)\\
\bottomrule
\end{tabular}
\end{center}



\subsection{Data for \texorpdfstring{$q=4$}{q4}}
Zeta-like tuples (restricted) of depth 2, weight $\leq q^3= 64$:

\begin{center}

\begin{tabular}{ *{6}{c} } \toprule
(1, 3)* & (1, 6)* & (1, 9)* & (1, 12)* & (1, 15)* & (2, 9)*\\
(2, 21) & (2, 27) & (3, 9)* & (3, 12)* &  (4, 27)* & (4, 39)*\\
(4, 51)* & (4, 57)* & (5, 18) & (5, 24) & (5, 27) &
(7, 24)\\
(7, 36) & (7, 39) & (7, 48)*& (7, 51) &
(7, 54) & (7, 57)*\\
(8, 39)* & (8, 51)* & (10, 51) & (11, 36)* & (11, 48)*
& (11, 51)*\\
(12, 51)* & (13, 48)* & (13, 51)* & (15, 48)* \\
\bottomrule
\end{tabular}
\end{center}

\begin{center}
Zeta-like tuples (restricted search) of depth 3 up to weight $q^3 = 64$:
\begin{tabular}{ *{4}{c} }\toprule
(1, 3, 12)* & (1, 6, 24) & (1, 12, 48) & (3, 12, 48)\\
(1, 12, 51) & (1, 15, 48) &  & \\ \bottomrule
 \end{tabular}
\end{center}


\pagebreak
\subsection{Data for \texorpdfstring{$q=5$}{q5}}
Zeta-like tuples (restricted) of depth 2,  weights $\leq 125$:
\begin{center}

\begin{tabular}{ *{6}{c} } \toprule
(1, 4)*   & (1, 8)*   & (1, 12)*  &  (1, 16)*  &  (1, 20)*  &  (1, 24)*\\
(2, 8)*   & (2, 12)*  & (2, 16)*  &  (2, 20)*  &  (3, 12)*  &   (3,16)*\\
(3, 20)*  & (4, 16)*  & (4, 20)*  &  (5, 44)*  &  (5, 64)*   &  (5, 68)*\\
(5, 84)*   & (5, 88)*   &  (5, 92)*  & (5, 104)*   &  (5, 108)*  & (5, 112)*\\
(5, 116)*  & (9, 40)   & (9, 60)   &  (9, 64)   & (9, 80)    & (9, 84)\\
(9, 88)   & (9, 100)*  &  (9, 104) & (9, 108)   &  (9, 112)  & (9, 116)*\\
(10, 64)* & (10, 84)*  & (10, 88)*  & (10, 104)*  & (10, 108)*  & (10, 112)*\\
(13, 60)  & (13, 80)  &  (13, 84) & (13, 100)*  & (13, 104)  & (13, 108)\\
(13, 112) & (14, 60)  & (14, 80)  & (14, 84)   & (14, 100)*  & (14, 104)\\
(14, 108) & (15, 84)* &  (15, 104)*&  (15, 108)* & (17, 80)   & (17, 100)*\\
(17, 104) & (17, 108) &  (18, 80) & (18, 100)*  & (18, 104)  & (19, 80)*\\
(19, 100)* & (19, 104)* &  (20, 104)*&  (21, 100)* & (21, 104)*  & (22, 100)*\\
(23, 100)* &  (24, 100)*\\
\bottomrule
\end{tabular}
\end{center}

\begin{center}
$q = 5$. Some zeta-like tuples.

\begin{tabular}{ *{4}{c} } \toprule
(1, 4, 20)* & (1, 20, 104) & (1, 24, 100) & (2, 20, 100)\\
(4, 20, 100) & (3, 20, 100) & (3, 20, 100, 500) &
(19, 100, 500)\\
(19, 100, 500, 2500)
\\
\bottomrule
\end{tabular}
\end{center}


\begin{center}
Summary of depth and weights classified by eulerian and zeta-like.
\begin{tabularx}{\linewidth}{ *{2}{c} X X} \hline
$q$ & depth & Eulerian weights  & Zeta-like weights\\ \hline
2 & 2 & 2, 3, 4, 7, 8, 15, 31, 63 &  \\
2 & 3 & 4, 7, 8, 15, 31, 63, 127 &  \\
2 & 4 & 8, 15, 16, 31, 63, 127, 255 &  \\
2 & 5 & 16, 31, 32, 63, 127, 255, 511 &  \\
2 & 6 & 32, 63, 64, 127, 255, 511, 1023 &  \\
\hline
3 & 2 & 6, 8, 26, 80 & 3, 5, 7, 9, 17, 23, 25, 27, 53,
71, 77, 79 \\
3 & 3 & 26, 80 & 9, 25, 27, 77, 79, 233, 239 \\
3 & 4 & 80, 242 & 27, 79, 81, 239, 241, 719 \\
\hline
4 & 2 & 12, 15, 63 & 4, 7, 10, 11, 13, 16, 23, 29, 31,
32, 43, 46, 47, 55, 58, 59, 61, 62, 64 \\
4 & 3 & 63 & 16, 31, 61, 64 \\
\hline
5 & 2 & 20, 24, 124 & 5, 9, 10, 13, 14, 15, 17, 18, 19,
21, 22, 23, 25, 49, 69, 73, 74, 89, 93, 94, 97, 98, 99, 109, 113, 114,
117, 118, 119, 121, 122, 123, 125 \\
5 & 3 & 124 & 25, 122, 123, 125, 619 \\
5 & 4 &  & 623, 3119 \\
\hline
\end{tabularx}
\end{center}

\newpage

{\bf Acknowledgments.}
The first author is grateful to the
 Centro de C\'omputo del Cuerpo Acad\'emico Modelado y Simulaci\'on
Computacional de Sistemas F\'isicos (UADY-CA-101) de la Universidad Aut\'onoma
de Yucat\'an as well as to the University of Arizona 
for allowing us  to use their computer servers for our computations.


\end{document}